\documentclass{article}

\usepackage{amssymb}

\usepackage{tikz}
\usetikzlibrary{snakes,arrows,automata}
\usepackage{pgf}
\usepackage{amssymb}
\usepackage{amsfonts}
\usepackage{latexsym}
\usepackage{amsthm}
\usepackage{xspace}
\usepackage{color}
\usepackage{amstext}
\usepackage{rotating}
\usepackage{amsmath}
\newcommand{\inty}{interpretability\xspace}

\newcommand{\formal}[1]{\ensuremath{{\sf {#1}}\xspace}}
\newcommand{\principle}[1]{\formal{#1}}
\newcommand{\eqbydef}{:=}

\newcommand{\con}{\bigwedge \hspace{-2.5mm} \bigwedge}
\newcommand{\dis}{\bigvee \hspace{-2.5mm} \bigvee}
\newcommand{\cone}[2]{\mathcal{C}_{#1}^{#2}}
\newcommand{\bis}[1]{{\ensuremath{\mathcal{S}_{#1}}}\xspace}
\newcommand{\depth}[1]{{\ensuremath{{\sf depth}(#1)}}\xspace}
\newcommand{\cebe}{{\ensuremath{\mathcal{C}_{\principle{B}}}}\xspace}

\newcommand{\ea}{\ensuremath{{\rm{EA}}}\xspace}
\newcommand{\pra}{\ensuremath{{\mathrm{PRA}}}\xspace}
\newcommand{\pa}{\ensuremath{{\mathrm{PA}}}\xspace}
\newcommand{\isig}[1]{{\ensuremath {\mathrm{I}\Sigma_{#1}}}\xspace}

\newcommand{\idel}[1]{{\ensuremath {\mathrm{I}\Delta_{#1}}}\xspace}

\newcommand{\ir}[1]{{\ensuremath{\mathrm{I}\Sigma^R_{#1}}}\xspace}

\newcommand{\il}{{\ensuremath{\textup{\textbf{IL}}}}\xspace}
\newcommand{\extil}[1]{\ensuremath{\textup{\textbf{IL}}{\sf\ensuremath{#1}}}\xspace}
\newcommand{\gl}{{\ensuremath{\textup{\textbf{GL}}}}\xspace}

\newcommand{\intl}[1]{{\ensuremath {\textup{\textbf{IL}}}({\rm #1})}}

\newcommand{\ilm}{\extil{M}}
\newcommand{\ilp}{\extil{P}}
\newcommand{\ilw}{\extil{W}}

\newcommand{\ilb}{\extil{B}}
\newcommand{\ilal}{\intl{All}\xspace}

\newtheorem{theorem}{Theorem}[section]

\newtheorem{lemma}[theorem]{Lemma}
\newtheorem{corollary}[theorem]{Corollary}

\newtheorem{definition}[theorem]{Definition}

\begin{document}



\title{Interpretability in \pra}

\author{Marta B\'{\i}lkov\'{a}, Dick de Jongh and Joost J. Joosten}
\date{2009}

\maketitle

%

\begin{abstract}
In this paper\footnote{We thank Lev Beklemishev for his help and suggestions. Evan Goris did a thorough proofread of an early draft and suggested a simplification of the notion of B-simulation. We thank Albert Visser for fruitful discussions and challenges.
We also thank Franco Montagna for his many contributions to the subject. 
Two unknown referees improved our paper considerably with their remarks and suggestions. Supported by grants GA CR 401/06/0387 and IAA900090703.}
from 2009 we study \intl{\pra}, the interpretability
  logic of \pra. As \pra is neither an essentially reflexive theory
  nor finitely axiomatizable, the two known arithmetical completeness
  results do not apply to \pra: \intl{\pra} is not \ilm or
  \ilp. \intl{\pra} does of course contain all the principles known to
  be part of \ilal, the interpretability logic of the principles
  common to all reasonable
  arithmetical theories. In this paper, we take two arithmetical
  properties of \pra and see what their consequences in the modal
  logic \intl{\pra} are. These properties are reflected in the
  so-called Beklemishev Principle \principle{B}, and Zambella's
  Principle \principle{Z}, neither of which is a part of \ilal. Both
  principles and their interrelation are submitted to a modal
  study. In particular, we prove a frame condition for
  \principle{B}. Moreover, we prove that \principle{Z} follows from a restricted form of
  \principle{B}. Finally, we give an overview of the known
  relationships 
  of \intl{\pra} to important other interpetability principles. 
\end{abstract}




\section{Introduction}

The notion of a relativized interpretation occurs in many places in
mathematics and in mathematical logic. If a theory $T$ interprets a
theory $S$, we shall write $T\rhd S$, which then, roughly, means
that there is a translation $\cdot^{t}$ from symbols in the language
of $S$ to formulas in the language of $T$ such that any theorem of $S$
becomes  a theorem of $T$ under the canonical extension of this
translation to formulas. In the notion of interpretation that we are
interested in, the logical structure of formulas has to be preserved
under the translation. Thus, for example,  
$(\varphi \vee \psi)^t = \varphi^t \vee \psi^t$ and in particular
$\bot^t=(\vee_{\emptyset})^t= \vee_{\emptyset} = \bot$. We refer the
reader to \cite{vis97}, \cite{DJ} and \cite{tars:unde53} for precise
definitions and examples.

In this paper, we shall not go 
much into the technical details of interpretations. Rather, we are
interested in the structural behavior of this notion of
interpretability. In particular, we are interested in the structural
behavior of interpretability on sentential extensions of a certain
base theory $T$. An easy example of such a structural property is the
transitivity of interpretations: 
\[
(T + \alpha \rhd T + \beta) \wedge (T + \beta \rhd T + \gamma) \rightarrow (T + \alpha \rhd T + \gamma).
\]
We can use so-called interpretability logics to capture, in a sense, the complete structural behavior of interpretability between sentential extensions of a certain base theory. We shall soon say a bit more on this. For now it is important to note that for a large collection of theories, the interpretability logic is known. 
\medskip

We call a theory \emph{reflexive} if it proves the consistency of any
of its finite sub-theories (as sets of axioms). We call a theory
\emph{essentially reflexive} if any finite sentential extension of it
is reflexive. It is easy to see that any theory with full induction,
like Peano Arithmetic, is essentially reflexive. The interpretability
logic of essentially reflexive theories was determined independently
by Berarducci and Shavrukov (\cite{bera:inte90},
\cite{shav:logi88}). We shall encounter this logic below under the
name of \ilm. The principle $(A\rhd B)\to (A\wedge\Box C\rhd
B\wedge\Box C)$ which is the particular feature of this system. It is
called Montagna's principle since it arose during the original discussions between
Franco Montagna and Albert Visser about the modal principles
underlying interpetability logic. It was known to Lindstr\"om and
\v{S}vejdar in arithmetic disguise before.

It turns out that theories which are finitely axiomatizable and which
contain a sufficient amount of arithmetic, have a different
interpretability logic which is called \ilp. In \cite{vis97}, the
first proof was given.

For no theory  that is neither finitely axiomatizable nor essentially reflexive, the interpretability logic is known. \pra is one such theory. In this paper, we shall make some first attempts to work out the interpretability logic of \pra.   

As such, this paper also fits into a larger project. As pointed out
above, different arithmetical theories have different interpretability
logics. A question that is open since a long time concerns the logic
of the core principles that pertain to \emph{all} reasonable
arithmetical theories - \ilal. As \pra is certainly a `reasonable
arithmetical theory', this core logic should also be a part of
\intl{\pra}. In this paper we shall not focus too much on the
principles in the core logic. Rather shall we consider the
interpretability behavior of \pra that is typical for this theory. 

One such principle that is characteristic for \pra is \emph{Beklemishev's} principle that shall be studied closely in this paper. This principle exploits the fact that any theory which is an extension of \pra by  $\Sigma_2$ sentences 
is reflexive. We give a characterization of this principle in terms of the modal semantics for interpretability logics.

\medskip

A topic that is closely related to interpretability logics, is that of
$\Pi_1$-conservativity logics. A theory $S$ is $\Pi_1$ conservative
over a theory $T$ in the same language of arithmetic, we shall write
$S\rhd_{\Pi_1}T$ whenever $S$ proves any $\Pi_1$ theorem that is
proven by $T$. In symbols: $T\vdash \pi \ \Longrightarrow S\vdash \pi$
for any $\pi \in \Pi_1$. It is easy to see that for any $\Sigma_1$
sentence $\sigma$, the following is a valid principle $S\rhd_{\Pi_1}T
\to S + \sigma \rhd_{\Pi_1}T + \sigma$. This principle is the basis for
Montagna's principle for interpretability logic, and  Beklemishev's
principle which is studied in this paper is a restriction of
Montagna's principle.

When $T$ and $S$ are both reflexive theories we have that $S\rhd T \
\leftrightarrow \ S\rhd_{\Pi_1}T$. This equivalence was exploited by H\'ajek and
Montagna who were the~first to show that the~$\Pi_1$-conservativity
logic of \pa is \ilm as well \cite{haje:cons90}. The observation about
the equivalence is more generally
important when looking at the repercussions of $\Pi_1$-conservativity
principles on interpretability logics. In this paper we shall consider
Zambella's principle for $\Pi_1$-conservativity logics and look at its
repercussions for the interpretability logic of \pra. We shall show
that Zambella does not add new information in the sense that its
modal-logical consequences are
already implied by Beklemishev's principle.  

It is remarkable that the notion of interpretability is, in a
sense, less stable than that of $\Pi_1$-conservativity.
H\'ajek and Montagna show that their results extends to all
reasonable theories containing \isig{1}. This was strengthened by
Beklemishev and Visser in \cite{Bekv04}: all theories extending the
parameter-free induction schema ${\mathrm I}\Pi_1^-$ have the same
$\Pi_1$-conservativity logic (\ilm) whereas in this range the interpretability
logics expose a~diverse and wild behavior. Note though that $\pra$ does not
prove ${\mathrm I}\Pi_1^-$, and, in fact, the $\Pi_1$-conservativity logic of \pra remains
unknown. 

A number of the results in this paper was first proved in \cite{jotesis}.


\section{Arithmetic}

Let us first fix some arithmetical notation. We use modal symbols
$\Box,\Diamond, \rhd$ both in modal and arithmetical statements, here
we fix their arithmetical meaning. We write, for an~arithmetical
sentence $\alpha$, $\Box_{\rm T}\alpha$ for formalized provability in
${\rm T}$, $\Box_{{\rm T},n}\alpha$ for formalized provability of
$\alpha$ in ${\rm T}$ using only non-logical axioms with G\"{o}del
numbers $\leq n$ and formulas of logical complexity $\leq n$ \footnote{Since \pra proves superexponentiation this is, \emph{in the case under study}, equivalent to the restriction of axioms to those $\leq n$}. Dually,
$\Diamond_{\rm T}\alpha = \neg\Box_{\rm T}\neg\alpha$ means formalized
consistency of $\alpha$ over 
${\rm T}$ (i.e.\ nonexistence of a~proof of a~contradiction from
$\alpha$), while
$\Diamond_{{\rm T},n}\alpha$ means $\neg\Box_{{\rm T},n}\neg\alpha$. 
For theories ${\rm T},{\rm S}$ we use ${\rm T}\rhd {\rm S}$ to denote
formalized interpretability of ${\rm S}$ in ${\rm T}$. For
arithmetical sentences $\alpha,\beta$, $\alpha\rhd_{\rm T}\beta$ means
${\rm T}+\alpha\rhd {\rm T} +\beta$. Similarly for theories ${\rm
  T},{\rm S}$, $\rhd_{\Pi_1}$ denotes formalized
$\Pi_1$-conservativity of ${\rm T}$ over ${\rm S}$ and for
arithmetical sentences $\alpha,\beta$, $\alpha\rhd_{\Pi_1}\beta$ means
${\rm T}+\alpha\rhd_{\Pi_1} {\rm T} +\beta$. 

\subsection{What is \pra?}

In the literature there are many definitions of \pra. Probably
the best known definition uses a language that contains a
function symbol for every primitive recursive function. The axioms
contain the defining equations of these functions. Moreover, there are
induction axioms for each $\Delta_0$-formula in this enriched
language.

Beklemishev has shown in \cite{Bek99b} that \pra is in a strong sense
equivalent (faithfully mutually interpretable) with
$(\ea)_{\omega}^2$. Here, $(\ea)_{\omega}^2$ is the theory that is
obtained by starting with \ea (= $\idel{0}+ {\sf exp}$) and iterating
`$\omega$ many times' $\Pi_2$-reflection. In symbols: $(\ea)^2_0=\ea$,
and $(\ea)^2_{n+1}= 
{\sf RFN}_{(\ea)_n^2}(\Pi_2)$.

In this paper, we shall use the definition:
\[
\pra := (\ea)_{\omega}^2.
\]
Under this definition, the following lemma is immediate.

\begin{lemma}\label{lemm:prarefl}
Any r.e.\ extension of \pra by $\Sigma^0_2$ sentences is reflexive.
\end{lemma}

\subsection{The Orey-H\'ajek Characterizations}

All theories that are mentioned here are supposed to be consistent and
have a poly-time recognizable axiomatization.  
Orey and H\'ajek have given several equivalent conditions on theories
which express that the one interprets the other. In this subsection we
shall briefly mention the one we shall need and refer to the
literature for proofs.   

\begin{lemma}
Whenever $T$ is reflexive we have that
\[
T \rhd S \ \ \Leftrightarrow \forall x\  T\vdash \neg\Box_{S,x}\bot 
\]
Moreover in the presence of
the totality of exponentiation this equivalence can be formalized. 
\[
 \vdash T \rhd S\leftrightarrow \forall x\ \Box_{T}\neg\Box_{S,x}\bot
\]


\end{lemma}

In \cite{jotesis} an overview is given of all the implications,
corresponding requirements and necessary arguments regarding Orey-H\'ajek. In
the~above Lemma the $\Leftarrow$ does not need the requirement of
reflexivity and can actually be formalized in ${\sf S}^1_2$. For the
other direction reflexivity is needed, and for its formalization,
the~totality of $\sf exp$ as well.

Note that, using the~above characterization, the~prima facie 
$\Sigma_3$ notion of interpretability becomes $\Pi_2$.

\section{Modal logics and semantics}

Similarly as formalized provability can be captured by modal
provability logic, we can use modal logic to reason about formalized
interpretability. Modal logic proved to be an~extremely useful tool to
reason about such formalized phenomena since it can visualize their
behaviour using a~simple language and an~intuitive frame
semantics. Perhaps the~most significant point where modal logic shows
its skills are completeness proofs - arithmatical completeness proofs
are based on modal completeness proofs obtained by rather standard
method of model theory of modal logics. For more on material contained
in this section we refer to \cite{vis97,jotesis,jogo04}. 

We will work with modal propositional language containing two
modalities - a~unary $\Box$ modality for provability and a~binary
$\rhd$ modality for interpretability. Modal interpretability formulas
are defined as follows: 
\[
 \mathcal{A}::=
 p\, |\, \bot \,|\, (\mathcal{A}\wedge\mathcal{A}) \,|\, (\mathcal{A}\to\mathcal{A}) \,|\, (\Box\mathcal{A}) \,|\, (\mathcal{A}\rhd\mathcal{A}) 
\]

We will use standard abbreviations $\Diamond, \vee, \neg, \top,
\leftrightarrow$, and we write $A\equiv B$ instead of $(A\rhd B)
\wedge (B\rhd A)$. We shall often omit brackets writing formulas. 
We say that $\neg, \Box,$ and $ \Diamond$ bind equally strong, they bind stronger then equally strong binding $\vee$ and $\wedge$ which in turn bind stronger then $\rhd$. The weakest binding connectives are $\to$ and $\leftrightarrow$.

\medskip

An~arithmetical interpretation of modal formulas is given by
\emph{arithmetical realizations}:  
for an~arithmetical theory $\mathrm{T}$, an~arithmetical
$\mathrm{T}$-realization is a~map $\ast$ sending propositional
variables $p$ to arithmetical sentences $p^*$. It is extended to
interpretability modal formulas as follows: first $\ast$ commutes with
all boolean connectives. Moreover $(\Box A)^* = \Box_T A^*$ and
$(A\rhd B)^* = A^*\rhd_T B^*$, i.e.\ $\ast\,$ translates modal operators
to formalized provability and interpretability over $\mathrm{T}$
respectively. 

An \emph{\inty principle} of an~arithmetical theory $\mathrm{T}$ is
a~modal formula $A$ such that $\forall\!\ast T\vdash A^\ast$. The
\emph{interpretability logic} of a~theory $\mathrm{T}$, denoted
\intl{T}, is then the~set of all the~interpretability principles of
$\mathrm{T}$. 

\subsection{The logic \il}

The~logic \il is in a~sense the~core interpretability logic - it is
a~(proper) part of the~interpretability logic of any reasonable
arithmetical theory: $\il\subset\intl{T}$. It captures the basic
structural behaviour of interpretability. 

\il is defined as the~smallest set of formulas containing all
propositional tautologies, all instantiations of the~following
schemata, and is closed under the Necessitation and Modus Ponens rules: 
\[
\begin{array}{ll}
\principle{L1} & \Box(A\to B)\to (\Box A\to\Box B) \\
\principle{L2} & \Box A\to\Box\Box A \\
\principle{L3} & \Box(\Box A\to A)\to \Box A \\
\principle{J1} & \Box(A\to B)\to A\rhd B \\
\principle{J2} & (A\rhd B)\wedge (B\rhd C)\to A\rhd C \\
\principle{J3} & (A\rhd C)\wedge (B\rhd C)\to A\vee B\rhd C \\
\principle{J4} & A\rhd B\to (\Diamond A\to\Diamond B) \\
\principle{J5} & \Diamond A\rhd A 
\end{array}
\]

Note that the~part of \il not containing the~$\rhd$ modality is
the~well-known G\"{o}del-L\"{o}b provability logic \gl, axiomatized
by the~first three schemata. It is easy to show that $\Box$ can be
defined in terms of $\rhd$ modality: $\vdash_{\il}\Box
A\leftrightarrow \neg A\rhd\bot$.  

\medskip

More interpretability logics are obtained extending \il by new
interpretability principles. Some of such principles are listed below:  

\[
\begin{array}{ll}
\principle{W}   & A\rhd B\to A\rhd B\wedge\Box\neg A \\
\principle{W^*} & A\rhd B\to B\wedge\Box C\rhd B\wedge\Box C\wedge\Box\neg A \\
\principle{M_0} & A\rhd B\to\Diamond A\wedge\Box C\rhd B\wedge\Box C \\
\principle{M}   & A\rhd B\to A\wedge\Box C\rhd B\wedge\Box C \\
\principle{P}   & A\rhd B\to\Box(A\rhd B) \\
\principle{R}   & A\rhd B\to \neg(A\rhd\neg C)\rhd B\wedge\Box C \\
\principle{R^*} & A\rhd B\to \neg(A\rhd\neg C)\rhd B\wedge\Box C\wedge\Box\neg A
\end{array}
\]

All of these principles are in \ilal except the principles
\principle{M} and \principle{P} which were mentioned above
already. For an overview, see~\cite{vis97} and~\cite{jogo04}. For the
last word on \ilal see~\cite{jogo08}.

For \principle{X} a set of principles we denote \extil{X} the~logic
extending \il with schemata from \principle{X}.  

There are some results considering arithmetical completeness of
interpretability logics: 
it was shown in \cite{bera:inte90},\cite{shav:logi88} that
the~interpretability logic of an~essentially reflexive theory (as
e.g. \pa) is \ilm.  For finitely axiomatizable theories containing
${\sf supexp}$ the~interpretability logic is known to be \ilp
(\cite{viss:inte90}).

An important consequence of \ilm that expresses
the $\Pi_1$-conservativity of interpretability more directly  is $(A\rhd
\Diamond B)\to\Box(A\to\Diamond B)$.

\subsection{Modal semantics}

Modal frame semantics of interpretability logics is based on \gl-frames extended with a~ternary accesibility relation interpreting the~binary $\rhd$ modality. The~ternary relation is however given by a~set of binary relations indexed by the~nodes:

\begin{definition} An \il-frame (a~Veltman frame) is a~triple $\langle W,R,S\rangle$ where $W$ is a~nonempty universe, $R$ is a~binary relation on $W$, and $S$ is a~set of binary relations on $W$, indexed by elements of $W$ such that 
\[
\begin{array}{ll}
1. & R\; \mbox{is transitive and conversely well-founded} \\
2. & yS_xz \Rightarrow xRy\, \&\, xRz \\
3. & xRy \Rightarrow yS_xy \\
4. & xRyRz \Rightarrow yS_xz \\
5. & uS_xvS_xw \Rightarrow uS_xw
\end{array}
\]

An \il-model is a~quadruple $\langle W,R,S,\Vdash\rangle$ where $\langle W,R,S\rangle$ is a~\il-frame and $\Vdash$ is a~subset of $W\times{\sf Prop}$, extending to boolean formulas as usualy and to modal formulas as follows:

\[
\begin{array}{lllll}
w &\Vdash&\Box A &\, \mbox{iff} \, & \forall v(wRv\Rightarrow v\Vdash A) \\
w &\Vdash&A\rhd B &\, \mbox{iff} \, & \forall u(wRu\; \&\; u\Vdash A\Rightarrow \exists v(uS_wv\Vdash B))
\end{array}
\]

\end{definition}

We adopt standard definitions of validity of a~modal formula in
a~model and in a~frame.  
Moreover, let \principle{X} be a~scheme of \inty logic. We say that
a~formula $\mathcal{C}$ in first or higher order logic is
a~\emph{frame condition} for \principle{X} if, for each frame $F$,
$$F\models\mathcal{C}\,\mbox{iff}\,F\models\principle{X}.$$  
\par\noindent\medskip
Let us list some known frame conditions (to be read universally quantified):
\[
 \begin{array}{ll}
\principle{M}   & xRyS_xzRu\Rightarrow yRu \\
\principle{M_0} & xRyRzS_xuRv\Rightarrow yRv \\
\principle{P}   & xRyRzS_xu\Rightarrow yRu\wedge zS_yu \\
\principle{W}   & (S_w;R)\, \mbox{is  conversely well-founded} \\
\principle{R}   & xRyRzS_xuRv\Rightarrow zS_yv 
 \end{array}
\]
We have the~following completeness results: 
\il is sound and complete w.r.t.\ (finite) \il frames, 
\ilp is complete w.r.t.\ (finite) \ilp frames (all in \cite{JoVe90}),
\ilw is complete w.r.t.\ (finite) \ilw frames (\cite{jonvelt99}, see
also \cite{jogo04}), \ilm is
complete w.r.t.\ (finite) \ilm frames (in \cite{JoVe90}, also in
\cite{bera:inte90}),

\section{Beklemishev's principle}

It is possible to write down a valid principle specific for the~\inty
logic of \pra. This was first done by Beklemishev (see \cite{vis97}).
Beklemishev's principle \principle{B} exploits the fact that
any finite $\Sigma_2$-extension of \pra is reflexive, together with the
fact that we have a good Orey-H\'ajek characterization for reflexive theories.

It turns out to be possible to define a class of modal formulae which are under
any arithmetical realization provably $\Sigma_2$ in \pra. These are called \emph{essentially $\Sigma_2$-formulas}, we write 
${\sf ES_2}$. Let us start by defining this class and some related classes.

The idea behind this definition is as follows. It is clear that each modal formula  that starts with a $\Box$ will become under any arithmetical realization a $\Sigma_1$ formula.
Likewise, taking Lemma 2 into account, we see that any formula of the form $A\rhd B$ where $A$ is $\Sigma_2$, will be under any arithmetical realization of complexity $\Pi_2$ and hence, 
$\neg(A\rhd B)$ will again be $\Sigma_2$. Note that we are here only formulating \emph{sufficient} conditions. It turns out to be rather tough to show these classes actually cover, up to provable equivalence, all formulae in the intended complexity class.

The class ${\sf BS}_1$ denotes the formulae that are boolean combinations of $\Sigma_1$ formulae ad thus certainly $\Delta_2$. Likewise, ${\sf ES}_3$ and ${\sf ES}_4$, stands for those modal formulae that are under any arithmetical realization always $\Sigma_3$ or $\Sigma_4$ respectively.

In our definition, $\mathcal{A}$ will stand for the set of all modal 
\inty formulae.
\[
\begin{array}{lll}
{\sf BS_1} 	&::= 	&\Box \mathcal{A} \mid \neg {\sf BS_1} \mid {\sf BS_1}\wedge
{\sf BS_1}\mid
{\sf BS_1}\vee {\sf BS_1} \ \\
{\sf ES_2}  	&::= 	&\Box \mathcal{A} \mid \neg \Box \mathcal{A} \mid {\sf ES_2}\wedge 
{\sf ES_2}\mid 
{\sf ES_2} \vee {\sf ES_2} \mid \neg ({\sf ES_2} \rhd \mathcal{A})\\
{\sf ES_3}  	&::= 	& \Box \mathcal{A} \mid \neg \Box \mathcal{A}  \mid 
{\sf ES_3}\wedge {\sf ES_3}\mid {\sf ES_3} \vee {\sf ES_3}\mid \mathcal{A}
\rhd \mathcal{A}\\
{\sf ES_4}	&::= 	& \Box \mathcal{A} 
\mid \neg {\sf ES_4} \mid {\sf ES_4} \wedge {\sf ES_4} \mid {\sf ES_4} \vee 
{\sf ES_4} \mid \mathcal{A}\rhd \mathcal{A}  \\
\end{array}
\]
For $n\geq 4$ we set ${\sf ES_n} := {\sf ES_4}$. We can now formulate Beklemishev's principle \principle{B}.
\[
\principle{B} \eqbydef A\rhd B \rightarrow A \wedge \Box C \rhd B\wedge \Box C \ \ \ \
\mbox{ for $A \in {\sf ES_2}$}
\]

Note that \principle{B} is just Montagna's principle \principle{M}
restricted to ${\sf ES_2}$-formulas. 

\begin{lemma}
$\bf{IL}\sf{B}\vdash\sf{B'}$, where
${\sf B'}: A\rhd B \rightarrow A\wedge C\rhd B\wedge C$ with
$A\in {\sf ES_{2}}$ and $C$ a~CNF $($a~conjunction of disjunctions$)$ of boxed formulas.
\end{lemma}

\begin{proof}
Easy.
\end{proof}

\section{Arithmetical soundness of \principle{B}}

By Lemma \ref{lemm:prarefl} we know that $\pra + \sigma$ is reflexive for any
$\Sigma_2(\pra)$-sentence $\sigma$. Thus, we get by Orey-H\'ajek
that
\begin{equation}\label{equa:prab}
\pra \vdash \sigma \rhd_{\pra}\psi \leftrightarrow
\forall x\ \Box_{\pra} (\sigma \rightarrow \Diamond_{\pra,x}\psi).
\end{equation}
Consequently, for $\sigma \in \Sigma_2(\pra)$, 
$\neg (\sigma \rhd_{\pra}\psi) \in \Sigma_2(\pra)$
and we see that, indeed, 
$\forall \, A{\in}{\sf ES_2}\, \forall *\ A^*\in \Sigma_2(\pra)$. 
This enables us to prove the arithmetical soundness of \principle{B}.

\begin{theorem}\label{theo:bsound}
For any formulas $B$ and $C$ we have that 
$\forall \, A {\in} {\sf ES_2}\, \forall * 
\pra \vdash (A\rhd B \rightarrow A \wedge \Box C \rhd B \wedge \Box C)^*$.
\end{theorem}

\begin{proof}
For some $A\in {\sf ES_2}$ and arbitrary $B$ and $C$, we consider
some realization $*$ and let $\alpha \eqbydef A^*$,
$\beta\eqbydef B^*$ and $\gamma \eqbydef C^*$. We reason in \pra and 
assume $\alpha\rhd_{\pra}\beta$. As $\alpha$ is $\Sigma_2(\pra)$, we get
by \eqref{equa:prab} that
\begin{equation}\label{equa:assoh}
\forall x\ \Box_{\pra}(\alpha \rightarrow \Diamond_{\pra,x}\beta).
\end{equation}
We now consider $n$ large enough (dependent on $\gamma$)
such that
\begin{equation}\label{equa:consoh}
\Box_{\pra}(\Box_{\pra}\gamma \rightarrow \Box_{\pra ,n}\Box_{\pra}\gamma).
\end{equation}
From general observations we have that, for large enough $n$,
\[
\Box_{\pra,n} (\delta \to \neg \epsilon) \wedge \Box_{\pra,n} \delta \to
\Box_{\pra,n} \neg \epsilon,
\]
whence
\begin{equation}\label{equa:restrictDistributivity}
\Diamond_{\pra,n}\epsilon \wedge \Box_{\pra,n} \delta \to \Diamond_{\pra,n}(\delta \wedge \epsilon) 
\end{equation}

Combining \eqref{equa:assoh}, \eqref{equa:consoh}, and using 
\eqref{equa:restrictDistributivity}, we see that for 
any $n$, 
$\Box (\alpha \wedge \Box \gamma \rightarrow \Diamond_{\pra,n}(\beta \wedge 
\Box \gamma))$. Clearly, $\alpha \wedge \Box \gamma$ is still a 
$\Sigma_2(\pra)$-sentence.\footnote{Actually, this observation is not necessary as we use the direction in the Orey-H\'ajek Characterization that does not rely on the reflexivity.} 
Again by \eqref{equa:prab} we get 
$\alpha \wedge \Box \gamma \rhd \beta \wedge \Box \gamma$.
\end{proof}

Let ${\sf M^{ES_n}}$ be the schema $A\rhd B \rightarrow 
A\wedge \Box C\rhd B\wedge \Box C$ with $A \in {\sf ES_n}$. Theorem \ref{theo:bsound} can be generalized using results of \cite{Bek97a} to the~theory \ir{n}, which is Robinson's arithmetic $\rm{Q}$ plus the~$\Sigma_n$ induction \emph{rule}, for $n=1,2,3$ as follows:

\begin{theorem}
$\intl{\ir{n}}\vdash \principle{M^{ES_{n+1}}}$ for $n=1,2,3$.
\end{theorem}


\section{A frame condition for \principle{B}}


Let us first fix some notation. 
If  $\mathcal{C}$ is a finite set, we write $xR\mathcal{C}$ as short for 
$\con_{c\in \mathcal{C}}xRc$. Similar conventions hold for the other relations.
The $A$-critical cone of $x$, $\cone{x}{A}$ is in this section defined as
$\cone{x}{A}:= \{ y \mid xRy \wedge \forall z\; (yS_xz \rightarrow z \not \Vdash A) \}$.

By $x{\uparrow}$ we denote the set of worlds that lie above $x$ w.r.t. the $R$ relation.
That is, $x{\uparrow}:= \{ y \mid xRy\}$.
With $yS_x\!{\uparrow}$ we denote the set of those $z$ for which
$yS_xz$.

We will consider frames both as modal models without a valuation and as structures for
first- 
(or sometimes 
second) order logic. We say that a model $M$ is based on
a frame $F$ if $F$ is precisely $M$ with the $\Vdash$ relation left out.
%
%




In this subsection we give 
the frame condition of Beklemishev's 
principle. 
Our frame condition holds on the class of finite frames. At first sight, the
condition might seem a bit awkward. On second sight it is just the
frame condition of \principle{M} with some simulation built in.
%
%
First we approximate the class ${\sf ES_2}$ by stages.
\begin{definition}  \ \\
$
\begin{array}{lll} 
{\sf ES_2^0} & := & {\sf BS_1} \\
{\sf ES_2^{n+1}} &:= & {\sf ES_2^n} \mid {\sf ES_2^{n+1}} \wedge {\sf ES_2^{n+1}} 
\mid {\sf ES_2^{n+1}} \vee {\sf ES_2^{n+1}} 
\mid \neg ({\sf ES_2^n} \rhd \mathcal{A})
\end{array}
$
\end{definition}
It is clear that ${\sf ES_2} = \cup_i {{\sf ES_2^i}}$. 
We now define some first order formulas $\bis{i} (b,u)$ that say that two nodes $b$ and $u$ in a frame
look alike. The larger $i$ is, the more the two points look alike. We
use the letter $\bis{}$ as to hint at a simulation.

\begin{definition}  \ \\  
$
\begin{array}{lll}
\bis{0}(b,u) & := & b{\uparrow} = u{\uparrow} \\
\bis{n+1} (b,u) & := & \bis{n} (b,u) \wedge \\
\ & \ & \forall c \; 
(bRc \rightarrow \exists c' \; (uRc' \wedge \bis{n}(c,c') 
\wedge c'S_u{\uparrow} \subseteq
cS_b{\uparrow} ))
\end{array}
$
\end{definition}

By induction on $n$ we easily see that $\forall n \; F\models \bis{n} (b,b) $ for all 
frames $F$ and all $b{\in}F$.
For $i\geq 1$ the relation $\bis{i} (b,u)$ is in general not symmetric.
However it is not hard to see that the $\bis{i}$ are transitive and reflexive.

\begin{lemma}\label{lemm:bekbis}
Let $F$ be a model. For all $n$ we have the following. If $F\models \bis{n}(b,u)$, then
$b\Vdash A \Rightarrow u\Vdash A$ for all $A{\in}{\sf ES_2^n}$.
\end{lemma}

\begin{proof}
We proceed by induction on $n$. If $n{=}0$, $A{\in}{\sf ES_2^0}$ can be written as
$\dis_i(\Box A_i \wedge \con_j \Diamond A_{ij})$.
Clearly, if $b{\uparrow}=u{\uparrow}$ then $b\Vdash A \Rightarrow u\Vdash A$.

Now consider $A{\in}{\sf ES_2^{n+1}}$ and $b$ and $u$ such that $F\models \bis{n+1} (b,u)$.
We can write 
\[
A = \dis_i (A_{i0} \wedge \con_{j \neq 0} \neg (A_{ij}\rhd B_{ij})),
\]
with $A_{ij}$ in ${\sf ES_2^n}$. If $b\Vdash A$, then for some $i$, 
$b\Vdash A_{i0} \wedge \con_{j \neq 0} \neg (A_{ij}\rhd B_{ij})$.
As $\bis{n+1} (b,u) \rightarrow \bis{n} (b,u)$, 
and by the induction hypothesis we
see that $u\Vdash A_{i0}$. So, we only need to see that 
$u\Vdash \neg (A_{ij} \rhd B_{ij})$ for $j{\neq}0$. As 
$b\Vdash \neg (A_{ij} \rhd B_{ij})$, for some $c\,{\in}\,\cone{b}{B_{ij}}$
we have $c\Vdash A_{ij}$. By $\bis{n+1} (b,u)$ we find a $c'$ such that
$uRc'$, and $c'S_u\!{\uparrow} \subseteq cS_b\!{\uparrow}$ (thus $cS_bc'$).
This guarantees that $c'{\in}\cone{u}{B_{ij}}$. Moreover we know that
$\bis{n}(c,c')$, thus by the induction hypothesis, as $c\Vdash A_{ij}$,
we get that $c'\Vdash A_{ij}$. Consequently 
$u\Vdash \neg (A_{ij} \rhd B_{ij})$.

\end{proof}

\begin{lemma}\label{lemm:beksent}
Let $F$ be a finite frame. For all $i$, and any $b{\in}F$, there is a valuation 
$V_i^b$ on $F$ and a formula $A_{i}^b{\in}{\sf ES_2^i}$ such that 
$F\models \bis{i} (b,u) \Leftrightarrow u\Vdash A_i^b$.
\end{lemma}

\begin{proof}
The proof proceeds by induction on $i$. First consider the basis case, that is, $i{=}0$.
Let $b{\uparrow}$ be given by the finite set
$\{ x_j\}_{j\in J}$. We define
\[
\begin{array}{lll}
y\Vdash p_j &\Leftrightarrow & y{=}x_j \\
y\Vdash r & \Leftrightarrow & bRy.
\end{array}
\]
Let $A_0^b$ be $\Box r \wedge \con_j \Diamond p_j$. It is now obvious that 
$u\Vdash A_0 \Leftrightarrow u{\uparrow}{=}b{\uparrow}$.

For the inductive step, we fix some $b$ and reason as follows.
First, let $V_i^b$ and $A_i^b$ be given by the induction hypothesis such
that $u\Vdash A_i^b \Leftrightarrow F\models \bis{i}(b,u)$. We do not specify 
the variables in $A_i$ but we suppose they do not coincide with any of the ones 
mentioned below. Let $b{\uparrow} = \{ x_j\}_{j\in J}$. The induction hypothesis 
gives us sentences $A_i^j$ (no sharing of variables) and valuations $V_i^j$
such that 
$F,u \Vdash A_i^j \Leftrightarrow F\models \bis{i} (x_j ,u)$.

Let $\{ q_j \}_{j\in J}$ be a set of fresh variables. $V_{i+1}^b$ will be 
$V_i^b$ and 
$V_i^j$ on the old variables. For the 
$\{ q_j \}_{j\in J}$ we define $V_{i+1}^b$ 
to act as follows:
\[
y\Vdash q_j \Leftrightarrow y \,{\not \in}\, x_j\!S_b{\uparrow}.
\]
Moreover we define
\[
A_{i+1}^b:= A_i^b \wedge \con_j \neg (A_i^j \rhd q_j).
\]
Now we will see that under the new valuation $V_{i+1}^b$,
\begin{itemize}
\item[$(\romannumeral 1 )$] 
$u\Vdash A_{i+1}^b \Rightarrow F\models \bis{i+1}(b,u)$,
\item[$(\romannumeral 2 )$]
$ F\models \bis{i+1}(b,u) \Rightarrow u\Vdash A_{i+1}^b$.
\end{itemize}
For $(\romannumeral 1 )$ we reason as follows. Suppose $u\Vdash A_{i+1}^b$. 
Then
also $u\Vdash A_i^b$ and thus $F\models \bis{i}(b,u)$. It remains to show that 
\[
F\models \forall c\; (bRc \rightarrow \exists c'\;
(uRc'  \wedge
\bis{i} (c,c') \wedge cS_bc' \wedge c'S_u\!{\uparrow}\subseteq cS_b\!{\uparrow})).
\]

To this purpose we consider and fix some $x_j$ in $b{\uparrow}$.
As $u\Vdash A_{i+1}^b$, we get that 
$u\Vdash \neg (A_i^j \rhd q_j)$. Thus, for some $c'\,{\in}\,\cone{u}{q_j}$,
$c'\Vdash A_i^j$. Clearly $c'\Vdash \neg q_j$ whence $x_jS_bc'$. Also 
$\forall t\ (c'S_uy \Rightarrow y\Vdash \neg q_j)$ 
which, by the definition of $V_{i+1}^b$
translates to $c'S_u{\uparrow} \subseteq x_jS_b{\uparrow}$. Clearly also $uRc'$.
By $c'\Vdash A_i^j$ and the induction hypothesis we get that 
$\bis{i}(x_j,c')$. Indeed we see that $F\models \bis{i+1}(b,u)$.

For $(\romannumeral 2 )$ we reason as follows.
As $F\models \bis{i+1}(b,u)$, also $F\models \bis{i}(b,u)$ and by the 
induction hypothesis, $u\Vdash A_i^b$. It remains to show that 
$u\Vdash \neg (A_i^j\rhd q_j)$ for any $j$. 
So, let us fix some $j$. Then, by the second part of 
the $\bis{i+1}$ requirement we find a $c'$ such that
\[
uRc' \wedge
\bis{i} (x_j,c') \wedge x_jS_bc' \wedge c'S_u\!{\uparrow}\subseteq x_jS_b\!{\uparrow}. 
\]
Now, $uRc' \wedge x_jS_bc' \wedge c'S_u\!{\uparrow}\subseteq
x_jS_b\!{\uparrow}$ gives us that 
$c'\,{\in}\,\cone{u}{q_j}$. By $\bis{i} (x_j,c')$ and the induction
hypothesis we get that  
$c'\Vdash A_i^j$. Thus indeed $u\Vdash \neg (A_i^j \rhd q_j)$.
\end{proof}

Note that in the proof of this lemma, we have only used conjunctions to construct
the formulas $A_i^b$.

\begin{definition}
For every $i$ we define the frame condition $\mathcal{C}_i$ to be
\[
\forall \, a,b \; (aRb \rightarrow \exists u \;
(bS_au \wedge \bis{i}(b,u) \wedge \forall \, d,e\;
(uS_adRe \rightarrow bRe))).
\]
\end{definition}

\begin{lemma}\label{lemm:partiallev}
Let $F$ be a finite frame. For all $i$, we have that
\begin{center}
for all 
$A{\in}{\sf ES_2^i}$,
$F\models A\rhd B \rightarrow A\wedge \Box C \rhd B \wedge \Box C$,\\
if and only if\\
$F\models \mathcal{C}_i$.
\end{center}
\end{lemma}

\begin{proof}
First suppose that $F\models \mathcal{C}_i$ and that 
$a\Vdash A\rhd B$ for some $A{\in}{\sf ES_2^i}$ and some valuation on $F$. 
We will show that 
$a\Vdash A\wedge \Box C \rhd B\wedge \Box C$ for any $C$.
Consider therefore some $b$ with $aRb$ and $b\Vdash A\wedge \Box C$.
The $\mathcal{C}_i$ condition provides us with a $u$ such that 
\[
bS_au \wedge \bis{i}(b,u) \wedge \forall \, d,e\;
(uS_adRe \rightarrow bRe) \ \ \ (*)
\]
As $F\models \bis{i}(b,u)$, we get by Lemma \ref{lemm:bekbis} that 
$u\Vdash A$. Thus, as $aRu$ and $a\Vdash A\rhd B$, we know
that there is some $d$ with $uS_ad$ and $d\Vdash B$. If now
$dRe$, by $(*)$, also $bRe$ and hence $e\Vdash C$. Thus, 
$d\Vdash B\wedge \Box C$. Clearly $bS_ad$ and thus
$a\Vdash A\wedge \Box C \rhd B\wedge \Box C$.

For the opposite direction we reason as follows.
Suppose that $F \not \models \mathcal{C}_i$. Thus, we can find 
$a,b$ with 
\[
aRb \wedge \forall u\; 
(bS_au \wedge \bis{i}(b,u) \rightarrow \exists \, d,e \; 
(uS_adRe \wedge \neg bRe)) \ \ \ (**).
\]
By Lemma \ref{lemm:beksent} we can find a valuation $V_i^b$ and a sentence
$A_i^b{\in}{\sf ES_2^i}$ such that 
$u\Vdash A_i^b\Leftrightarrow F\models \bis{i}(b,u)$. Let $q$ and $s$ be 
fresh variables. Moreover, let $\mathcal{D}$ be the following set.
\[
\mathcal{D} := \{  d{\in} F \mid
bS_adRe \wedge \neg bRe \mbox{ for some $e$ }\} .
\]
We define a valuation $V$ that is an extension of $V_i^b$ by 
stipulating that
\[
\begin{array}{lll}
y \Vdash q & \leftrightarrow & (y{\in} \mathcal{D}) \vee \neg (bS_ay), \\
y\Vdash s &\leftrightarrow & bRy.  \\
\end{array}
\]
We now see that 
\begin{itemize}
\item[$(\romannumeral 1 )$] 
$a\Vdash A_i^b \rhd q $,
\item[$(\romannumeral 2 )$]
$a\Vdash \neg (A_i^b \wedge \Box s \rhd q \wedge \Box s)$.
\end{itemize}
For $(\romannumeral 1 )$ we reason as follows. Suppose that 
$aRb'$ and $b'\Vdash A_i^b$. If $\neg (bS_a b')$, $b'\Vdash q$ and we are 
done. So, we consider the case in which $bS_ab'$. 
As $\bis{i}(b,b')$, $(**)$ now yields us a
$d{\in}\mathcal{D}$ such that 
$b'S_ad$. Clearly $bS_ad$ and thus, by definition, $d\Vdash q$.

To see $(\romannumeral 2 )$ we notice that $b\Vdash A_i^b \wedge \Box s$.
But if $bS_ay$ and $y\Vdash q$, by definition $y{\in} \mathcal{D}$ and 
thus $y\Vdash \neg \Box s$. Thus $b{\in}\cone{a}{q\wedge \Box s}$ and
$a\Vdash \neg (A_i \wedge \Box s \rhd q \wedge \Box s)$.
\end{proof}

The following theorem is now an immediate corollary of the above reasoning.

\begin{theorem}
A finite frame $F$ validates all instances of Beklemishev's principle 
if and only if $\forall i\; F\models \mathcal{C}_i$.
\end{theorem}

\begin{definition}\label{defn:b_i}
Let \principle{B_i} be the principle
$A\rhd B \rightarrow A \wedge \Box C \rhd B \wedge \Box C$ for 
$A \in {\sf ES_2^i}$.
\end{definition}

\begin{corollary}
For a finite frame we have 
$F\models \principle{B_i} \Leftrightarrow F \models \mathcal{C}_i$.
\end{corollary}

For the class of finite frames, we can get rid of the
universal quantification in the frame condition of Beklemishev's 
principle. Remember that
\depth{x}, the depth of a point $x$, is the length of the longest
chain of $R$-successors starting in $x$.
\begin{lemma}
If $\bis{n} (x,x')$, then $\depth{x}=\depth{x'}$.
\end{lemma}
\begin{proof}
$\bis{n} (x,x') \Rightarrow \bis{0} (x,x') \Rightarrow x{\uparrow}=x'{\uparrow}$.
\end{proof}
\begin{lemma}\label{lemm:justonesimulation}
If $\bis{n}(x,x') \ \& \ \depth{x}\leq n$, then $\bis{m}(x,x')$ for all $m$.
\end{lemma}
\begin{proof}
The proof goes by induction on $n$. For $n=0$, the result is clear. So,
we consider some $x,x'$ with $\bis{n+1}(x,x') \ \& \ \depth{x} \leq n+1$.
We are done if we can show $\bis{m+1}(x,x')$ for $m\geq n+1$.

This, we prove by a subsidiary induction on $m$. The basis is trivial.
For the inductive step, we assume $\bis{m}(x,x')$ for some 
$m\geq n+1$ and set out to prove $\bis{m+1}(x,x')$, that is
\[
\bis{m}(x,x') \wedge \forall y\ 
(xRy \rightarrow \exists y'\ 
(yS_xy' \wedge \bis{m}(y,y') \wedge y'S_{x'}{\uparrow}\subseteq yS_x{\uparrow}))
\]
The first conjunct is precisely the induction hypothesis. For the second
conjunct we reason as follows.
As $m\geq n+1$, certainly $\bis{n+1} (x,x')$.
We consider $y$ with $xRy$.
By $\bis{n+1}(x,x')$, we find a $y'$ with
\[
yS_xy' \wedge \bis{n}(y,y') \wedge y'S_{x'}{\uparrow}\subseteq yS_x{\uparrow} .
\]
As $xRy$ and $\depth{x} \leq n+1$, we see $\depth{y}\leq n$. Hence by the 
main induction, we get that $\bis{m}(y,y')$ and we are done.
\end{proof}
\begin{definition}
A \principle{B}-simulation on a frame is a binary relation \bis{} for which the 
following holds.
\begin{enumerate}
\item
$\bis{} (x,x') \rightarrow x{\uparrow} = x'{\uparrow}$

\item
$\bis{} (x,x') \ \& \ xRy \rightarrow \exists y' 
(yS_x y' \wedge \bis{} (y,y') \wedge y'S_{x'}{\uparrow}\subseteq yS_x{\uparrow})$

\end{enumerate}
\end{definition}
If $F$ is a finite frame that satisfies $\mathcal{C}_i$ for all $i$, we 
can consider $\bigcap_{i\in \omega} \bis{i}$. This will certainly be
a \principle{B}-simulation. 
\begin{definition}
The frame condition \cebe is defined as follows. $F\models \cebe$
if and only if there is a \principle{B}-simulation \bis{} on $F$ such that
for all $x$ and $y$,
\[
xRy \rightarrow \exists y' (yS_xy' \wedge \bis{} (y,y') \wedge 
\forall d,e\ 
(y'S_xdRe \rightarrow yRd)).
\]
\end{definition}
An immediate consequence of Lemma \ref{lemm:justonesimulation} is
the following theorem.
\begin{theorem}
For $F$ a finite frame, we have 
\[
F\models \principle{B} \ \  \Leftrightarrow \ \  F \models \cebe .
\]
\end{theorem}
\noindent
Note that the \principle{M}-frame condition can be seen as a special 
case of the frame condition of \principle{B}: we demand that \bis{} be
the identity relation.

It is not hard to see that the frame condition of $\principle{M_0}$ follows from
$\mathcal{C}_0$. And indeed, $\ilb \vdash \principle{M_0}$ as
$\Diamond A \in {\sf ES_2}$ and $A\rhd B \rightarrow \Diamond A \rhd B$.
Actually, we have that $\extil{B_1}\vdash \principle{M_0}$.

\section{Beklemishev and Zambella}


\newcommand{\edtwo}{{\sf BS_1}}


Zambella proved in \cite{Zam94} a fact concerning $\Pi_1$-consequences of theories
with a 
$\Pi_2$ 
axiomatization.
As we shall see, his result has some repercussions on the 
study of the \inty logic of \pra.
\begin{lemma}[Zambella]\label{lemm:zambie}
Let $T$ and $S$ be two theories axiomatized by $\Pi_2$-axioms.
If $T$ and $S$ have the same $\Pi_1$-consequences then $T+S$ 
has no more $\Pi_1$-consequences than $T$ or $S$.
\end{lemma}
\noindent
In \cite{Zam94}, Zambella gave a model-theoretic proof of this lemma. 
As was sketched by G.\ Mints 
(see \cite{Bekv04}), also a finitary proof based on Herbrand's theorem
can be given. This proof can certainly be formalized in the 
presence of the superexponentiation function, thus it yields a principle for the
$\Pi_1$-conservativity logic of $\Pi_2$-axiomatized theories. We
denote it here as \principle{Z(EP_2^c)}. 
\[
{\sf Z(EP_2^c)} \ \ \ (A\equiv_{\Pi_1} B) \rightarrow A\rhd_{\Pi_1} A \wedge B  \ \ \mbox{ for $A$ and $B$ in 
${\sf EP_2^c}$.}
\]

where the~class  ${\sf EP_2^c}$ of modal formulas is defined as follows:

\[
\begin{array}{lll}
{\sf EP_2^c} & ::=  &\Box  \mathcal{A} \mid \neg \Box \mathcal{A} \mid
{\sf EP_2^c} \wedge 
{\sf EP_2^c} \mid {\sf EP_2^c} \vee {\sf EP_2^c}  \mid \mathcal{A} \rhd \mathcal{A}. \\
\end{array}
\]
The class ${\sf EP}_2^c$ is of course tailored so that any arithmetical realization will be provably $\Pi_2$. Note that the superscript $c$ is there to indicate that the $\rhd$ modality is to be interpreted as a formalization of the notion of $\Pi_1$ conservativity. It is not hard to see that the formalization of this notion is itself $\Pi_2$.  Moreover, note that this class coincides in extension with the earlier defined class ${\sf ES}_3$.

Since PRA is $\Pi_2$ axiomatized and proves totality of the supexp function the
principle ${\sf  Z(EP)}_2^c$ applies to PRA.

But there are
repercussions for the~interpretability logic of \pra as well. We know
that for reflexive theories  $\Pi_1$-conservativity coincides with
\inty. We also know that any $\Sigma_2$-extension  
of \pra is reflexive (Lemma \ref{lemm:prarefl}). Altogether this means
  that a~statement $\alpha\rhd\beta$ and $\alpha\rhd_{\Pi_1}\beta$ are
  equivalent if $\alpha$ is in $\Sigma_2$ and $\pra + \alpha$ is
  $\Pi_2$-axiomatized, i.e. $\alpha$ is in $\Delta_2$. 

We arrive at Zambella's principle for \inty logic:
\[
\principle{Z} \ \ \ (A \equiv B) \rightarrow A \rhd A \wedge B \ \ \ 
\mbox{for $A$ and $B$ in $\sf BS_1$}
\]
For the $\Pi_1$-conservativity logic of \pra, the principle 
\principle{Z(EP_2^c)} is really informative (see \cite{Bekv04}), it is
the~only principle known on top of the~basic ones for
the~$\Pi_1$-conservativity logic of \pra.  The~principle \principle{Z}
for interpretability logic is very interesting as well but it does
turn out to be derivable in \ilb as we will now proceed to show. (See
however the final remark of this section.) 

Here modal logic again proves to be informative - to have such a~proof
is interesting since it is not at all clear to us how the~two principles
relate arithmetically.  

We shall give a purely syntactical proof of $\extil{B_0} \vdash {\sf
  Z}$, $\principle{B_0}$ being a~restriction of $\principle{B}$ to
${\sf BS_1}$ formulas, see Definition \ref{defn:b_i}. The proof
  in~\cite{jotesis} of the same fact was not correct.

Throughout the~proof we consider a~full disjunctive normal form of modal formulas:

\begin{definition}
A~\emph{full disjunctive normal form} (a~full DNF) over a~finite set
of formulas $\{C_1,\ldots,C_n\}$ is a~disjunction of conjunctions of
the~form $\pm C_1\wedge\ldots\wedge\pm C_n$ where $+C_i$ means $C_i$
and $-C_i$ means $\neg C_i$, i.e., each $C_i$ occurs either positively
or negatively in each disjunct. 
\end{definition}

Each propositional formula is clearly equivalent to a~formula in full
DNF over the~set of propositional atoms occurring in it. Similarly
each modal ${\sf \sf BS_1}$-formula, being a~boolean combination of
boxed formulas, is equivalent to a formula in full DNF over the~set of
its boxed subformulas, or even over any finite set of boxed formulas
containing its boxed subformulas (or just its boxed
subforumulas maximal w.r.t.\ box-depth).  

\begin{theorem}
$\extil{B_0} \vdash \principle{Z}$
\end{theorem}

\begin{proof}

Let $A,B \in {\sf BS_1}$ and let $\{A_1,\ldots,A_m\}$ be the~set of boxed subformulas of \emph{both} $A$ and $B$. Assume w.l.o.g. that $A$ and $B$ are in full DNF over $\{A_1,\ldots,A_m\}$. Assume $A\equiv B$. We show that $A\rhd A\wedge B$. Since $A$ comes in full DNF, this means to show, for each disjunct $D$ of $A$, that $D\rhd A\wedge B$. In fact, we show this for any disjunct of $A$ or $B$.

A~disjunct $D$ of either $A$ or $B$ is fully determined by the~set $D^{\Box}$ of boxed formulas occurring positively in it. We shall write $D^{\Box}$ also for the~conjunction of its members.
\medskip\par\noindent
We first show, if $D$ is a member of $A$ or $B$ which has a~maximal set $D^{\Box}$ (no disjunct $E$ with $E^{\Box}$ properly containing $D^{\Box}$ occurs in $A$ or $B$) then $D\rhd A\wedge B$:

Suppose such $D$ is in $A$, the~other case is symmetrical. Since $D\rhd A$ we have also $D\rhd B$. Then, noting that $D^{\Box}$ is a~conjunction of boxed formulas and applying $\principle{B_0}$, we obtain $D\rhd B\wedge D^{\Box}$.

Now take any disjunct $E$ of $B$ for which $E^{\Box}$ does not contain $D^{\Box}$. Then $E$ contradicts $D^{\Box}$ by its negative part. We distinguish two cases: if for all $E$ in $B$ the~set $E^{\Box}$ does not contain $D^{\Box}$, then $B$ contradicts $D^{\Box}$. It follows from  $D\rhd B\wedge D^{\Box}$ that $D\rhd\bot$. Then clearly $D\rhd A\wedge B$.

Otherwise $B$ does contain $E$ with $E^{\Box}$ containing $D^{\Box}$. But since $D$ has a maximal Box-set, $E$ and $D$ must be the same and $D$ occurs in $B$ as well. Thus $D\rhd B\wedge D$ and, since $\vdash D\to A$, also $D\rhd A\wedge B$.\\
We have shown that all maximal disjuncts interpret $A\wedge B$.
\medskip\par\noindent
We show by induction that the~same is true for all other disjuncts of $A$ and $B$. This suffices for the~proof.

Assume that, for all $k'$ with $m\geq k'>k$ and all disjuncts  $D$ in either $A$ or $B$ with $D^{\Box}$ of size $k'$, $D\rhd A\wedge B$ (this has already been shown for $k$ equal to the~size of the~maximal Box-set in $A$ and in $B$ which is certainly less then $m$). Consider a~disjunct $D$ of $A$, the other case is again symmetrical. Assume w.l.o.g. that $D^{\Box}$ has size $k$. We have to show $D\rhd A\wedge B$:\\
Since $D\rhd A$ and hence $D\rhd B$, we again have that $D\rhd B\wedge
D^{\Box}$. Now $D^{\Box}$ conflicts with all the~disjuncts of $B$,
Box-set of which is not a~superset of $D^{\Box}$. Again, we
distinguish two cases: if there are no disjuncts of $B$  with a
Box-set which is a superset of $D^{\Box}$ then $B$ conflicts with
$D^{\Box}$ and $D\rhd\bot$ and thus $D\rhd A\wedge B$. 

Otherwise some disjuncts of $B$ do have a Box-set which is a superset
of $D^{\Box}$. Let $E_1,\ldots,E_l$ be all such disjuncts of
$B$. Then, since $D\rhd B\wedge D^{\Box}$ and $\vdash B\wedge
D^{\Box}\to E_1\vee\ldots\vee E_l$ (where $E_1\vee\ldots\vee E_l$ is
the~part of $B$ not conflicting with $D^{\Box}$), we obtain $D\rhd
E_1\vee\ldots\vee E_l$. Now it suffices to show that each $E_i$
interprets $A\wedge B$. 

Fix an $E_i$ and suppose $E_i^{\Box}$ have size $k$. But then $E_i = D$ and thus we have, as before, $D\rhd (B\wedge D)\rhd (B\wedge A)$.
If $E_i^{\Box}$ have size greater then $k$, the~induction hypothesis apply and we obtain that $E_i$
interprets $A\wedge B$.

\end{proof}

Actually it is possible to
extend Zambella's principle somewhat in such a way that it is no
longer clear whether the result is still derivable from
\principle{B}. First note that the formulas in ${\sf ES_2}$ are just
the propositional
combinations of $\Box$-formulas. 

Zambella's principle for interpretability logic as studied in this paper reads
\[
A \equiv B \rightarrow A\rhd A \wedge B
\]
where $A$ and $B$ should both be ${\sf BS}_1$. 
However, to have access to the ideas behind Zambella's principle, it is sufficient that $A$ and $B$ be both provably of complexity $\Delta_2$. We can thus look at those $\sf{ES}_2$ formulae who are provably equivalent to the negation of some other $\sf{ES}_2$ formula and plug those formulae in. Reflecting this thought in a formula yields\footnote{We would like to thank one of the referees for pointing out that our original extension of Zambella's principle for interpetability logic could actually be even generalized to its current form.}
\[
\Box ((A \leftrightarrow A')\wedge (B\leftrightarrow B')) \to (A \equiv B \rightarrow A\rhd A \wedge B)
\]
where $A$, $A'$, $B$ and $B'$ are all from $\sf{ES}_2$. It actually makes sense to call this principle \emph{the} Zambella principle for interpretability logic as it more precisely reflects the arithmetical ingredients. We have chosen not to do so as to be consistent with earlier papers.

%
%
%

\section{Delimitation of \intl{\pra}}

Let us see what we can conclude about \intl{\pra} from
the~above. Certainly \intl{\pra} includes \ilal but it is more than
that because \principle{B} is not a principle of \ilal. The latter is clear
from the~fact that $\ilal\subseteq\ilm\cap\ilp$ and \principle{Z} is
not in \ilp: consider the~following model: 

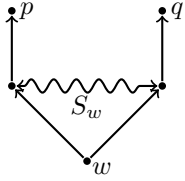
\begin{figure}[ht]

\begin{center}

\tikzstyle{circl}=[circle, fill=black!100,thick,inner sep=0pt,minimum
  size=1mm]

\tikzstyle{circu}=[circle, draw=black!100,thick,inner sep=0pt,minimum
  size=1mm]

\begin{tikzpicture}

\pgfputat{\pgfxy(0.2,-0.1)}{\pgfbox[center,center]{\bf $w$}};
\pgfputat{\pgfxy(-0.8,2)}{\pgfbox[center,center]{\bf $p$}};
\pgfputat{\pgfxy(1.2,2)}{\pgfbox[center,center]{\bf $q$}};

\node[circl] (1) at ( 0,0) {};

\node[circl] (2) at ( -1,1) {};

\node[circl] (3) at ( 1,1) {};

\node[circl] (4) at ( -1,2) {};

\node[circl] (5) at ( 1,2) {};

\draw [->,thick] (1) -- (2);

\draw [->, thick] (1) -- (3);

\draw [->,thick] (2) -- (4);

\draw [->, thick] (3) -- (5);

\draw [<->,thick,snake=snake, segment length=3mm, line after
  snake=1mm, line before snake=1mm] (2) -- (3)

node [below,text width=3cm,text centered,midway]{$S_w$};

\end{tikzpicture}

\end{center}


\end{figure}

\noindent We have $w\Vdash\Diamond p\equiv\Diamond q$ and $w\nVdash p\rhd p\wedge
q$, thus Zambella fails. The model is clearly an~\ilp model.  

This shows, by derivability of \principle{Z} from \principle{B}, that
indeed \principle{B} is not a principle of \ilal. 

Also we know that \intl{\pra} is not \ilm since \principle{M} is not
in \intl{\pra}, as A.\ Visser discusses in \cite{vis97}: the~two logics
cannot be the~same because if \ilm is a~part of the~\inty logic of
a~theory then it is a~part of the~\inty logic of any of its finite
extensions as well. This cannot be the~case for \pra because not all
of its finite extensions are reflexive. A~more specific example of
a~principle of \ilm which is not in \intl{\pra} can be given:  
\[
A\rhd \Diamond B \to \Box(A\rhd\Diamond B).
\]
That this~formula is not in \intl{\pra} can be shown using Shavrukov's
result from \cite{Sha97} about complexity of the~set
$\{\psi|\psi\in\Pi_1\;\&\;\phi\rhd\psi\}$; see \cite{vis97} for
the~full proof. 

\medskip

We know that $\principle{M_0}$ is provable in \ilb. The~other
principles surely contained in \intl{\pra} are \principle{B}, 
\principle{R} and \principle{W} (\principle{R^*} is the conjunction of
\principle{R} and \principle{W}). Let us show they are mutually
independent. Note that for nonderivability proofs soundness
suffices. 

First let us recall the~frame conditions for the~two principles $\principle{W}$ and $\principle{R}$. The~condition for $\principle{W}$ requires that the~composition $(S_w;R)$ is conversely well-founded, the~condition for $\principle{R}$ is the~following: $xRyRzS_xuRv\Rightarrow zS_yv$.

\begin{description}
 \item[\principle{W} vs.\ \principle{B}:] It is easy to see that
 $\principle{W}\nvdash\principle{B}$ since the~former is in \ilal while
 the~later is not in it. Since $\principle{R}$ is in \ilal as well, $\principle{W}, \principle{R}\nvdash\principle{B}$. The~following frame
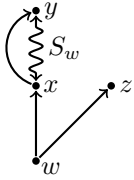
\begin{figure}[ht]
\begin{center}
\tikzstyle{circl}=[circle, fill=black!100,thick,inner sep=0pt,minimum size=1mm]
\tikzstyle{circu}=[circle, draw=black!100,thick,inner sep=0pt,minimum size=1mm]
\begin{tikzpicture}[segment aspect=.2]
\pgfputat{\pgfxy(0.2,-0.1)}{\pgfbox[center,center]{\bf $w$}};
\pgfputat{\pgfxy(0.2,1)}{\pgfbox[center,center]{\bf $x$}};
\pgfputat{\pgfxy(0.2,2)}{\pgfbox[center,center]{\bf $y$}};
\pgfputat{\pgfxy(1.2,1)}{\pgfbox[center,center]{\bf $z$}};

\node[circl] (1) at ( 0,0) {};
\node[circl] (2) at ( 0,1) {};
\node[circl] (3) at ( 0,2) {};
\node[circl] (4) at ( 1,1) {};

\draw [->,thick] (1) -- (2);
\draw [->,thick] (1) -- (4);
\draw [<-,thick] (3) ..controls (-0.5,1.8) and (-0.5,1.2).. (2);
\draw [<->,thick,snake=snake,line after snake=1mm, line before snake=1mm, segment length=2mm] (2) -- (3)
node [right,text width=0.5cm,text centered,midway]
{$S_w$};

\end{tikzpicture}
\end{center}
\end{figure}

\noindent is an~\ilb frame and it violates the~frame
condition for \principle{W}: $wRxRy$ and $xS_wyS_wx$ and $wRz$. Now
$z$ is bi-similar to $y$ and \principle{B} is ensured. Thus $\principle{B}\nvdash\principle{W}$.

Moreover, the same frame, being an $\principle{R}$ frame, shows that $\principle{B}, \principle{R}\nvdash\principle{W}$: the~only case to check is $wRxRyS_w xRy$. Now the~condition for $\principle{R}$ requires $yS_x y$, but this is clearly the~case since $S_x$ is reflexive over $x$.
\\ \medskip
\item[\principle{R} vs.\ \principle{B}:] Again, since
  $\principle{R}\in\ilal$, it cannot be that
  $\principle{R}\vdash\principle{B}$. We have already discussed that neither $\principle{R}, \principle{W}\vdash\principle{B}$. The~following frame 
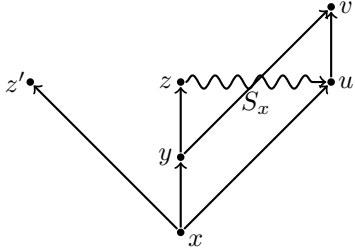
\begin{figure}[h]
\begin{center}
\tikzstyle{circl}=[circle, fill=black!100,thick,inner sep=0pt,minimum size=1mm]
\tikzstyle{circu}=[circle, draw=black!100,thick,inner sep=0pt,minimum size=1mm]
\begin{tikzpicture}
\pgfputat{\pgfxy(0.2,-0.1)}{\pgfbox[center,center]{\bf $x$}};
\pgfputat{\pgfxy(-2.2,2)}{\pgfbox[center,center]{\bf $z'$}};
\pgfputat{\pgfxy(-0.2,1)}{\pgfbox[center,center]{\bf $y$}};
\pgfputat{\pgfxy(-0.2,2)}{\pgfbox[center,center]{\bf $z$}};
\pgfputat{\pgfxy(2.2,2)}{\pgfbox[center,center]{\bf $u$}};
\pgfputat{\pgfxy(2.2,3)}{\pgfbox[center,center]{\bf $v$}};

\node[circl] (1) at ( 0,0) {};
\node[circl] (2) at ( -2,2) {};
\node[circl] (3) at ( 0,1) {};
\node[circl] (4) at ( 0,2) {};
\node[circl] (5) at ( 2,2) {};
\node[circl] (6) at ( 2,3) {};

\draw [->,thick] (1) -- (2);
\draw [->, thick] (1) -- (3);
\draw [->,thick] (3) -- (4);
\draw [->, thick] (1) -- (5);
\draw [->, thick] (5) -- (6);
\draw [->, thick] (3) -- (6);
\draw [->,thick,snake=snake, segment length=3mm, line after snake=1mm] (4) -- (5)
node [below,text width=3cm,text centered,midway]
{$S_x$};
\end{tikzpicture}
\end{center}
\end{figure}

\noindent is an \ilb-frame violating the~frame condition of
\principle{R}: We have a~basic situation violating $\principle{R}$, which is $xRyRzS_xuRv$ and
$\neg zS_yv$. To ensure \principle{B} for $y$ we add an~arrow $yRv$,
to ensure \principle{B} for $z$, we add a~bi-similar world $z'$ such
that $xRz'$ and $z'$ has no successors at all.  

Moreover, since the frame is clearly a~$\principle{W}$ frame as well, we have shown that $\principle{B},\principle{W}\nvdash\principle{R}$.
\\ \medskip
\item[\principle{R} vs.\ \principle{W}:] already discussed in \cite{jogo04}.
 \end{description}
 
It is clear from our exposition that, though we have solved a number
of problems concerning 
\intl{\pra},  many remain open, e.g.\ those connected with our incomplete
knowledge of 
\ilal. Also, we lack a modal completeness theorem for \ilb. Unfortunately,
the complexity of the frame condition for \principle{B} makes this
seem an intractable problem at the present time. In any case, the
logic of interpetability is far from being a~finished 
subject.


\end{document}